\documentclass[reqno,a4paper]{amsart}
\usepackage{amssymb}
\usepackage{amsmath}
\newcommand{\half}{\frac{1}{2}}
\newcommand{\R}{\mathbb R}
\begin{document} 
\newtheorem{prop}{Proposition}[section]
\newtheorem{Def}{Definition}[section] \newtheorem{theorem}{Theorem}[section]
\newtheorem{lemma}{Lemma}[section] \newtheorem{Cor}{Corollary}[section]

\title[Chern-Simons-Higgs and Chern-Simons-Dirac]{\bf The Chern-Simons-Higgs and the Chern-Simons-Dirac equations in Fourier-Lebesgue spaces}
\author[Hartmut Pecher]{
{\bf Hartmut Pecher}\\
Fakult\"at f\"ur Mathematik und Naturwissenschaften\\
Bergische Universit\"at Wuppertal\\
Gau{\ss}str.  20\\
42119 Wuppertal\\
Germany\\
e-mail {\tt pecher@math.uni-wuppertal.de}}
\date{}

\begin{abstract}
The Chern-Simons-Higgs and the Chern-Simons-Dirac systems in Lorenz gauge are locally well-posed in suitable Fourier-Lebesgue spaces $\hat{H}^{s,r}$. Our aim is to minimize $s=s(r)$ in the range $1<r \le 2$ . If $r \to 1$ we show that we almost reach the critical regularity dictated by scaling. In the classical case $r=2$ the results are due to Huh and Oh. Crucial is the fact that the decisive quadratic nonlinearities fulfill a null condition. 
\end{abstract}
\maketitle
\renewcommand{\thefootnote}{\fnsymbol{footnote}}
\footnotetext{\hspace{-1.5em}{\it 2000 Mathematics Subject Classification:} 
35Q40, 35L70 }
\normalsize 
\setcounter{section}{0}
\section{Introduction and main results}
\noindent We consider the Chern-Simons-Higgs system (CSH) in the Minkowski space 
${\mathbb R}^{1+2} = {\mathbb R}_t \times {\mathbb R}_x^2$ with metric ${\eta}_{\mu \nu} = diag(1,-1,-1)$ :
\begin{align}
\label{0.1}
 F_{\mu \nu} & =  \frac{2}{\kappa} \epsilon_{\mu \nu \lambda} Im(\overline{\phi} D^{\lambda} \phi) \\
\label{0.2}
D_{\mu} D^{\mu} \phi & = - \phi V'(|\phi|^2) \, ,
\end{align}
with initial data
\begin{equation}
\label{0.3}
A_{\mu}(0) = a_{\mu} \, , \, \phi(0) = f \, , \, (\partial_t \phi)(0) = g \, , 
\end{equation}
and the Chern-Simons-Dirac system (CSD)
\begin{align}
\label{0.4}
F_{\mu\nu} & = -\frac{2}{\kappa} \epsilon_{\mu \nu \lambda} (\overline{\psi} \gamma^{\lambda} \psi) \\
\label{0.5}
i \gamma^{\mu} D_{\mu} \psi - m\psi &= 0 
\end{align}
with initial data
\begin{equation}
\label{0.6}
 \psi(0) = \psi_0 \quad , \quad  A_{\mu}(0) = a_{\mu} \, , 
\end{equation}
where we use the convention that repeated upper and lower indices are summed, Greek indices run over 0,1,2 and Latin indices over 1,2. Here 
\begin{align*}
D_{\mu}  & := \partial_{\mu} - iA_{\mu} \\
 F_{\mu \nu} & := \partial_{\mu} A_{\nu} - \partial_{\nu} A_{\mu} 
\end{align*}
$F_{\mu \nu} : {\mathbb R}^{1+2} \to {\mathbb R}$ denotes the curvature, $\phi : {\mathbb R}^{1+2} \to {\mathbb C}$ is a scalar field, $\psi$ is a column vector with 2 complex components, $\psi^{\dag}$ is the complex conjugate transpose of $\psi$ and $\overline{\psi} := \psi^{\dag}\gamma^0$ and $A_{\nu} : {\mathbb R}^{1+2} \to {\mathbb R}$ are the gauge potentials.  $\kappa > 0$ is the Chern-Simons coupling constant, $m \ge 0$ is the mass of the spinor field $\psi$. We use the notation $\partial_{\mu} = \frac{\partial}{\partial x_{\mu}}$, where we write $(x^0,x^1,...,x^n) = (t,x^1,...,x^n)$ and also $\partial_0 = \partial_t$ and $\nabla = (\partial_1,\partial_2)$. $\epsilon_{\mu \nu \lambda}$ is the totally skew-symmetric tensor with $\epsilon_{012} = 1$, and a typical Higgs potential is $V(r) = \kappa^{-2} r(1-r)^2$. 
The gamma matrices $\gamma^{\mu}$ are defined by
$$ \gamma^0 = \sigma^3 \, , \, \gamma^1 = i \sigma^2 \, , \, \gamma^2 = -i \sigma^1 \, , $$
where the Pauli matrices $\sigma^j$ are given by \\
$$ \sigma^1 = \left( \begin{array}{cc}
0 & 1  \\ 
1 & 0  \end{array} \right)  \, 
 , \, \sigma^2 = \left( \begin{array}{cc}
0 & -i  \\
i & 0  \end{array} \right) \, ,  \, \sigma^3 = \left( \begin{array}{cc}
1 & 0  \\
0 & -1  \end{array} \right) \, .$$  

The CSH-model was proposed by Hong, Kim and Pac \cite{HKP} and Jackiw and Weinberg \cite{JW} in the study of vortex solutions in the abelian Chern-Simons theory.

The CSD-model was introduced by Cho, Kim and Park \cite{CKP} and Li-Bhaduri \cite{LB}.

The equations are gauge invariant. 
The most common gauges are the Coulomb gauge $\partial^j A_j =0$ , the Lorenz gauge $\partial^{\mu} A_{\mu} = 0$ and the temporal gauge $A_0 = 0$. In this paper we exclusively study the Lorenz gauge.

The Fourier-Lebesgue spaces $\widehat{H}^{s,r}$ are defined as the completion of ${\mathcal S}(\R^2)$ with respect to the norm $\|f\|_{\widehat{H}^{s,r}} = \| \langle \xi \rangle^s \widehat{f}\|_{L^{r'}}$, where $\frac{1}{r} + \frac{1}{r'} = 1$, and $\widehat{f}$ denotes the Fourier transform of $f$.\\[0.2em]

The aim is to minimize the regularity of the data so that local well-posedness holds. Persistence of higher regularity is then a consequence of the fact that the results are obtained by a Picard iteration.

Our results for both systems turn out to be almost optimal with respect to scaling for $r$ close to $1$  (as shown at the end of this section).\\[0.2em]

We now formulate the main result for the CSH system.
\begin{theorem}
\label{Thm.0.1}
Let $1 < r \le 2$ and $s > \frac{3}{2r} - \half$ . Assume
$$ a_\mu \in \widehat{H}^{s,r}(\R^2) \, , \, f \in \widehat{H}^{s+\half,r}(\R^2) \, , \, g \in \widehat{H}^{s+\half,r}(\R^2) $$
satisfy the constraint equation
$$ \partial_1 a_2 - \partial_2 a_1 = \frac{2}{\kappa} Im(\overline{f}(g-ia_0f)) \, . $$
Then there exists $ T > 0$ , $T=T(\|a_{\mu}\|_{\widehat{H}^{s,r}} , \|f\|_{\widehat{H}^{s+\half,r}} , \|g\|_{\widehat{H}^{s-\half,r}})$ such that the CSH system (\ref{0.1}),(\ref{0.2}),(\ref{0.3}) under the Lorenz gauge $\partial^{\mu} A_{\mu} =0$ has a unique solution
$$ A_{\mu} \in X^r_{s,b,+}[0,T] + X^r_{s,b,-}[0,T] \, , \, \phi \in X^r_{s+\half,b,+}[0,T] + X^r_{s+\half,b,-}[0,T] ,$$
$$ \partial_t \phi \in X^r_{s-\half,b,+}[0,T] + X^r_{s-\half,b,-}[0,T] , $$
where $b= \half+\frac{1}{2r}+$ . This solution satisfies
$$ A_{\mu} \in C^0([0,T],\widehat{H}^{s,r}) \, ,  \phi \in C^0([0,T],\widehat{H}^{s+\half,r}) \, ,  \partial_t \phi \in C^0([0,T],\widehat{H}^{s-\half,r}) \, . $$
\end{theorem}
The spaces $X^r_{s,b,\pm}$ are generalizations of the Bourgain-Klainerman-Machedon spaces $X^{s,b}$ (for $r=2$). We define $X^r_{s,b\pm}$ as the completion of ${\mathcal S}(\R^{1+2})$ with respect to the norm
$$ \|\phi\|_{X^r_{s,b\pm}} := \| \langle \xi \rangle^s \langle \tau \pm |\xi| \rangle^b \tilde{\phi}(\tau,\xi)\|_{L^{r'}_{\tau \xi}} $$
for $1 \le r \le 2$, $\frac{1}{r} + \frac{1}{r'}=1$ , where $\, \tilde{} \,$ denotes the Fourier transform with respect to space and time.\\[0.5em]

For the CSD system the main result is the following.
\begin{theorem}
\label{Thm.0.2}
Let $1 < r \le 2$ and $s > \frac{3}{2r} - \half$ . Assume
$$ a_\mu \, , \, \psi \in \hat{H}^{s,r}(\R^2) $$
satisfy the constraint equation
$$ \partial_1 a_2 - \partial_2 a_1 = - \frac{2}{\kappa} (\psi_0 ^{\dag} \psi_0) \, . $$
Then there exists $T>0$ , $T=T(\|a_{\mu}\|_{\widehat{H}^{s,r}} , \|\phi\|_{\widehat{H}^{s,r}}) $ such that the CSD system (\ref{0.4}),(\ref{0.5}),(\ref{0.6}) under the Lorenz gauge $\partial^{\mu} A_{\mu} =0$ has a unique solution
$$ A_{\mu} , \psi \in X^r_{s,b,+}[0,T] + X^r_{s,b,-}[0,T] \,,$$
where $b= \half+\frac{1}{2r}+$ . 
This solution satisfies
$$ A_{\mu},\psi \in C^0([0,T],\widehat{H}^{s,r}) \, . $$
\end{theorem}

{\bf Remark:} As these results are obtained by a Picard iteration it is well-known that continuous dependence on the data and persistence of higher regularity hold. \\[0.5em]

In the case $r=2$ , $s > \frac{1}{4}$ , $b=\frac{3}{4}+$ these results were obtained by Huh and Oh \cite{HO} for the CSH and the CSD system in Lorenz gauge. Their article is the bases of the present paper. They lowered down the regularity assumptions on the data improving earlier local well posedness results for the CSD system by Huh \cite{H} who had to assume $a_{\mu} \in H^{\half} $ , $\psi_0 \in H^{\frac{5}{8}}$ in Lorenz gauge, and $a_{\mu} \in L^2$ , $\psi_0 \in H^{\half+}$ in Coulomb gauge. In Coulomb gauge local well-posedness was shown for $\psi_0 \in H^{\frac{1}{4}+}$ by Bournaveas-Candy-Machihara \cite{BCM} and by a  different method by the author \cite{P}, who also proved local well-posedness in temporal gauge for $\psi_0 \in H^{\frac{3}{8}+}$ .

In the CSH case Selberg and Tesfahun \cite{ST} obtained global well-posedness in Lorenz gauge under a sign condition on $V$ for data $a_{\mu} \in \dot{H}^{\half}$ , $\phi \in H^1$ , $\partial_t \phi \in L^2$ , which is the energy regularity. They also obtained a local solution $A_{\mu} \in C^0([0,T],H^{\frac{3}{8}+})$ ,  $\phi \in C^0([0,T],H^{\frac{7}{8}+}) \cap C^1([0,T],H^{-\frac{1}{8}+})$. In temporal gauge global well-posedness in energy space and above was shown by the author \cite{P1}.

Most of these results used the fact that some or all quadratic nonlinear terms satisfy null conditions, which is also crucial for our paper. We also rely on bilinear estimates in $X^{s,b}$ - spaces by Foschi and Klainerman \cite{FK} and d'Ancona-Foschi-Selberg \cite{AFS}.

Well-posedness problems in Fourier-Lebesgue spaces $\hat{H}^{s,r}$ were first considered by Vargas-Vega \cite{VV} for 1D Schr\"odinger equations. Gr\"unrock showed LWP for the modified KdV equation \cite{G}, a result which was improved by Gr\"unrock and Vega \cite{GV}. Gr\"unrock treated derivative nonlinear wave equations in 3+1 dimensions \cite{G1} and obtained an almost optimal result as $r \to 1$ with respect to scaling using calculations by Foschi and Klainerman \cite{FK}. Later Grigoryan-Tanguay \cite{GT} gave similar results in 2+1 dimensions based on bilinear estimates by Selberg \cite{S}. 

Systems of nonlinear wave equations in the 2+1 dimensional case for nonlinearities which fulfill a null condition were considered by Grigoryan-Nahmod \cite{GN}. This paper is fundamental for the present paper, because the CSH as well as the CSD equations in Lorenz gauge are systems of a similar form.

The paper is organized as follows. In Chapter 2 we start by formulating the general local well-posedness (LWP) theorem for systems of nonlinear wave equations, as it was given by Gr\"unrock \cite{G}. Then estimates for the product of two solutions of the linear wave equation in Fourier-Lebesgue spaces are given, based mainly on Foschi-Klainerman \cite{FK}, who treated the $L^2$-based case. Moreover we also rely on bilinear estimates by d'Ancona-Foschi-Selberg \cite{AFS} in the $L^2$-case and its generalization to the general case by Grigoryan-Tanguay \cite{GT}. In addition, we have to estimate cubic nonlinearities. In Chapter 3 the proof of LWP for the Chern-Simons-Higgs system is given, following Huh-Oh \cite{HO}, who considered the $L^2$-case, and proved that the system can be reformulated as a system of nonlinear wave equations, which fulfill a null condition either directly or by duality. For the necessary estimates we rely on the bilinear estimates given in Chapter 2. In Chapter 4 the LWP result for the Chern-Simons-Dirac system is proven. We use 
the results by d'Ancona-Foschi-Selberg \cite{AFS1} for the case of the Dirac-Klein-Gordon equations concerning the Dirac part, and Huh-Oh \cite{HO} again when reformulating the system as a system of nonlinear wave equations with nonlinearities which fulfill a null condition. The necessary estimates which were given by Huh-Oh for the $L^2$-case are then implied by the generalizations in Chapter 2.

The CSH system is invariant under the scaling
$$ \phi_{\lambda}(t,x) = \lambda^{\half} \phi(\lambda t, \lambda x) \, , \,  A_{\mu}^{\lambda}(t,x)= \lambda A_{\mu}(\lambda x,\lambda t) \, . $$
Under this scaling the norms of the initial data satisfy
\begin{align*}
\| \phi_{\lambda}(0,\cdot)\|_{\dot{\hat{H}}^{r,s+\half}(\R^2)}& = \lambda^{1+s-\frac{2}{r}} \|f\|_{\dot{\hat{H}}^{r,s+\half}(\R^2)}  \\
\| A_{\mu}^{\lambda}(0,\cdot)\|_{\dot{\hat{H}}^{r,s}(\R^2)}& = \lambda^{1+s-\frac{2}{r}} \|a_{\mu}\|_{\dot{\hat{H}}^{r,s+\half}(\R^2)} \, .
\end{align*}
Therefore the scaling critical exponent is $s_c = \frac{2}{r} -1$ for $\phi$ and $A_{\mu}$. For $r=2$ we have $s_c =0$, so that the result of Huh-Oh \cite{HO} is $1/4$ away from it, whereas for $r=1+$ we have $s_c = 1-$,  so that for $r$ close to $1$ Theorem \ref{Thm.0.1} is optimal up to the endpoint.

The CSD system is invariant under the scaling
$$ \psi_{\lambda}(t,x) = \lambda \psi(\lambda t, \lambda x) \, , \,  A_{\mu}^{\lambda}(t,x)= \lambda A_{\mu}(\lambda x,\lambda t) \, , $$
so that
\begin{align*}
\| \psi_{\lambda}(0,\cdot)\|_{\dot{\hat{H}}^{r,s(\R^2)}}& = \lambda^{1+s-\frac{2}{r}} \|\psi_0\|_{\dot{\hat{H}}^{r,s}(\R^2)}  \\
\| A_{\mu}^{\lambda}(0,\cdot)\|_{\dot{\hat{H}}^{r,s}(\R^2)}& = \lambda^{1+s-\frac{2}{r}} \|a_{\mu}\|_{\dot{\hat{H}}^{r,s+\half}(\R^2)} \, .
\end{align*}
The critical exponent is $s_c = \frac{2}{r} -1$ for $\phi$ and $A_{\mu}$. Similar as for the CSH system we have the result of Huh-Oh in the case $r=2$, which is $1/4$ away from the critical regularity and the almost optimal result for $r=1+$ in Theorem \ref{Thm.0.2}.

\section{Bilinear estimates}
We start by collecting some fundamental properties of the solution spaces. We rely on \cite{G}. The spaces $X^r_{s,b,\pm} $ with norm  $$ \|\phi\|_{X^r_{s,b\pm}} := \| \langle \xi \rangle^s \langle \tau \pm |\xi| \rangle^b \tilde{\phi}(\tau,\xi)\|_{L^{r'}_{\tau \xi}} $$ for  $1<r<\infty$ are Banach spaces with ${\mathcal S}$ as a dense subspace. The dual space is $X^{r'}_{-s,-b,\pm}$ , where $\frac{1}{r} + \frac{1}{r'} = 1$. The complex interpolation space is given by
$$(X^{r_0}_{s_0,b_0,\pm} , X^{r_1}_{s_1,b_1,\pm})_{[\theta]} = X^r_{s,b,\pm} \, , $$
where $s=(1-\theta)s_0+\theta s_1$, $\frac{1}{r} = \frac{1-\theta}{r_0} + \frac{\theta}{r_1}$ , $b=(1-\theta)b_0 + \theta b_1$ . Similar properties has the space $X^r_{s,b}$ , defined by its norm 
$$ \|\phi\|_{X^r_{s,b}} := \| \langle \xi \rangle^s \langle |\tau| - |\xi| \rangle^b \tilde{\phi}(\tau,\xi)\|_{L^{r'}_{\tau \xi}} \, . $$ 
We denote $X^2_{s,b}$ by $H^{s,b}$ .
We also define
$$ X^r_{s,b,\pm}[0,T] = \{ u = U_{|[0,T]\times \R^2} \, : \, U \in X^r_{s,b,\pm} \} $$
with
$$ \|u\|_{X^r_{s,b,\pm}[0,T]} := \inf \{ \|U\|_{X^r_{s,b,\pm}} : U_{|[0,T]\times \R^2} = u \} $$
and similarly $X^r_{s,b}[0,T]$ . \\
If $u=u_++u_-$, where $u_{\pm} \in X^r_{s,b,\pm} [0,T]$ , then $u \in C^0([0,T],\hat{H}^{s,r})$ , if $b > \frac{1}{r}$ .

The "transfer principle" in the following proposition, which is well-known in the case $r=2$, also holds for general $1<r<\infty$ (cf. \cite{GN}, Prop. A.2 or \cite{G}, Lemma 1). We denote $ \|u\|_{\hat{L}^p_t(\hat{L}^q_x)} := \|\tilde{u}\|_{L^{p'}_{\tau} (L^{q'}_{\xi})}$ .
\begin{prop}
\label{Prop.0.1}
Let $1 \le p,q \le \infty$ .
Assume that $T$ is a bilinear operator which fulfills
$$ \|T(e^{\pm_1 itD} f_1, e^{\pm_2itD} f_2)\|_{\hat{L}^p_t(\hat{L}^q_x)} \lesssim \|f_1\|_{\hat{H}^{s_1,r}} \|f_2\|_{\hat{H}^{s_2,r}}$$
for all combinations of signs $\pm_1,\pm_2$ , then for $b > \frac{1}{r}$ the following estimate holds:
$$ \|T(u_1,u_2)\|_{\hat{L}^p_t(\hat{L}^q_x)} \lesssim \|u_1\|_{X^r_{s_1,b}}  \|u_2\|_{X^r_{s_2,b}} \, . $$
\end{prop}

The general local well-posedness theorem is the following (cf. \cite{G}, Theorem 1).
\begin{theorem}
\label{Theorem0.3}
Let $N(u)$ be a nonlinear function of degree $\alpha > 0$.
Assume that for given $s \in \R$, $1 < r < \infty$ there exist $ b > \frac{1}{r}$ and $b'\in (b-1,0)$ such that the estimates
$$ \|N(u)\|_{X^r_{s,b',\pm}} \le c \|u\|^{\alpha}_{X^r_{s,b,\pm}} $$
and 
$$\|N(u)-N(v)\|_{X^r_{s,b',\pm}} \le c (\|u\|^{\alpha-1}_{X^r_{s,b,\pm}} + \|v\|^{\alpha-1}_{X^r_{s,b,\pm}}) \|u-v\|_{X^r_{s,b,\pm}} $$
are valid. Then there exist $T=T(\|u_0\|_{\hat{H}^{s,r}})>0$ and a unique solution $u \in X^r_{s,b,\pm}[0,T]$ of the Cauchy problem
$$ \partial_t u \pm iDu = N(u) \quad , \quad u(0) = u_0 \in \hat{H}^{s,r} \, , $$
where $D$ is the operator with Fourier symbol $|\xi|$. This solution is persistent and the mapping data upon solution $u_0 \mapsto u$ , $\hat{H}^{s,r} \to X^r_{s,b,\pm}[0,T_0]$ is locally Lipschitz continuous for any $T_0 < T$.
\end{theorem}

The following proposition relies on estimates given by Foschi-Klainerman \cite{FK}.

\begin{prop}
\label{Prop.1.1}
Let $u$ and $v$ be solutions of the linear wave equation $\Box\, u = \Box \,v =0$ in $\R^{2+1}$ with $u(0)=f$ , $\partial_t u(0)=0$ , $v(0)=g$ , $\partial_t v(0) = 0$. Assume $1 < r \le 2$ , $\alpha_1+\alpha_2-\alpha_0 = \gamma + \frac{1}{r}$ , $\alpha_0 > \frac{1}{r}-\gamma$ , $0 \le \alpha_0 \le \alpha_1,\alpha_2$ , $\alpha_1 + \alpha_2 > \frac{2}{r}$ , $ \gamma \ge \max(\frac{1}{2r},\alpha_1-\frac{1}{r},\alpha_2-\frac{1}{r})$ and $\max(\alpha_1,\alpha_2) \neq \frac{3}{2r}$ . Then the following estimate holds
$$ \|{\mathcal F}(D^{\alpha_0} B^{\gamma}_{\pm}(u,v))\|_{L^{r'}_{\tau \xi}} \lesssim \|\widehat{D^{\alpha_1} f}\|_{L^{r'}_{\xi}} \|\widehat{D^{\alpha_2} g}\|_{L^{r'}_{\xi}} \, , $$
where
$$ \widehat{(B^{\gamma}_{\pm}(f,g))}(\xi) := \int b^{\gamma}_{\pm}(\xi,\eta) \widehat{f}(\eta) \widehat{g}(\xi - \eta) d\eta $$
with 
$$b_+(\xi,\eta) = |\eta| + |\xi - \eta| - |\xi| \quad , \quad b_-(\xi,\eta) = |\xi| - ||\eta|-|\xi-\eta|| \, . $$
\end{prop}
\begin{proof}
We decompose $uv = u_+ v_+ + u_+ v_-+u_- v_+ +u_- v_-$ , where $u_{\pm} (t) = e^{\pm itD} f$ and $v_{\pm}(t) = e^{\pm itD} g$ . It suffices to consider $u_+ v_+$ and $u_+ v_-$ . Using
$$ \tilde{u}_{\pm}(\tau,\xi) = c \delta(\tau \mp |\xi|) \widehat{f}(\xi) \, , \, \tilde{v}_{\pm}(\tau,\xi) = c \delta(\tau \mp |\xi|) \widehat{g}(\xi) $$
we have to estimate
\begin{align}
\nonumber
&\| \int |\xi|^{\alpha_0} b^{\gamma}_{\pm}(\xi,\eta) \delta(\tau - |\eta| \mp |\xi-\eta|) \widehat{f}(\eta) \widehat{g}(\xi - \eta) d\eta\|_{L^{r'}_{\tau \xi}} \\
\label{b}
& = \| \int |\xi|^{\alpha_0} ||\tau|-\xi||^{\gamma} \delta(\tau - |\eta| \mp |\xi-\eta|) \widehat{f}(\eta) \widehat{g}(\xi - \eta) d\eta\|_{L^{r'}_{\tau \xi}} \, ,
\end{align}
where the equality holds by the definition of $b_{\pm}$. By H\"older's inequality we obtain
\begin{align*}
& \left| ||\tau|-|\xi||^{\gamma} |\xi|^{\alpha_0} \int \delta(\tau-|\eta|\mp |\xi-\eta|) \widehat{f}(\eta) \widehat{g}(\xi-\eta) d\eta \right|^{r'} \\
& \lesssim  ||\tau|-|\xi||^{\gamma r'}  |\xi|^{\alpha_0 r'} | \int \delta(\tau-|\eta|\mp |\xi-\eta|) |\eta|^{-\alpha_1 r} |\xi - \eta|^{-\alpha_2 r} d\eta
 |^{\frac{r'}{r}} \\
&\quad \cdot \int |\eta|^{\alpha_1 r'} |\widehat{f}(\eta)|^{r'}
|\xi-\eta|^{\alpha_2 r'}|\widehat{g}(\xi-\eta)|^{r'} d\eta  \, .
\end{align*}
Thus
\begin{align*}
& \left\| ||\tau|-|\xi||^{\gamma} |\xi|^{\alpha_0} \int \delta(\tau-|\eta|\mp |\xi-\eta|) \widehat{f}(\eta) \widehat{g}(\xi-\eta) d\eta \right\|^{r'}_{L^{r'}_{\tau\xi}} \\
& \lesssim \sup_{\tau,\xi} I \, \,
\| \widehat{D^{\alpha_1}f}\|_{L^{r'}} \|\widehat{D^{\alpha_2}g}\|_{L^{r'}} \, ,
\end{align*}
where
$$ I:= ||\tau|-|\xi||^{\gamma}  |\xi|^{\alpha_0} | (\int \delta(\tau-|\eta|\mp |\xi-\eta|) |\eta|^{-\alpha_1 r} |\xi - \eta|^{-\alpha_2 r} d\eta
 )^{\frac{1}{r}} \, . $$
It remains to show
$ \sup_{\tau,\xi} I \lesssim 1 \, . $ Consider first the  (+,+)-case. By \cite{FK}, Prop. 4.3 we obtain
$$ \int \delta(\tau-|\eta|-|\xi - \eta|) |\eta|^{-\alpha_1 r} |\xi - \eta|^{-\alpha_2 r} d\eta  \sim \tau^A ||\tau|-|\xi||^B \, $$
where $A=\max(\alpha_1 r,\alpha_2r, \frac{3}{2}) - \alpha_1 r - \alpha_2 r$ and $B=1 -\max(\alpha_1 r,\alpha_2r, \frac{3}{2}) \, .$
Let w.l.o.g. $\alpha_1 \ge \alpha_2$.

If $\alpha_1 < \frac{3}{2r}$ we obtain $A= \frac{3}{2} - (\alpha_1+\alpha_2)r$ , $B=-\half$. Thus  we obtain
$$
I^r \lesssim |\xi|^{\alpha_0 r} ||\tau|-|\xi||^{\gamma r} \tau^{\frac{3}{2}-(\alpha_1 + \alpha_2)r} ||\tau|-|\xi||^{-\half} 
 \lesssim \tau^{\alpha_0 r + (\gamma r - \half)+\frac{3}{2}-(\alpha_1+\alpha_2)r} = 1 \, ,
$$
where we used $\tau = |\xi - \eta| + |\eta| \ge |\xi|$ , $\gamma \ge \frac{1}{2r}$ and $ \alpha_1+\alpha_2-\alpha_0 = \gamma + \frac{1}{r} $.

If $\alpha_1 > \frac{3}{2r}$ we obtain $A=-\alpha_2 r$ and $B=1-\alpha_1 r$, so that
$$ I^r \lesssim |\xi|^{\alpha_0 r} ||\tau|-|\xi||^{\gamma r +1-\alpha_1 r} \tau^{-\alpha_2 r} \lesssim \tau^{\alpha_0 r + \gamma r +1-(\alpha_1 + \alpha_2)r} = 1 \, ,$$
under our assumption $\gamma \ge \alpha_1-\frac{1}{r}$ .

Next we consider the (+,--) - case. We distinguish two cases:\\
{\bf a.} $|\eta| + |\xi-\eta| \le 2|\xi|$ . By \cite{FK}, Prop. 4.5 we obtain
$$ \int_{||\eta|+|\xi-\eta| \le 2|\xi|} \delta(\tau-|\eta|+|\xi-\eta|) |\eta|^{-\alpha_1 r} |\xi - \eta|^{-\alpha_2 r} d\eta \sim |\xi|^A ||\xi|-|\tau||^B \, . $$
Here $A= \max(\alpha_2 r,\frac{3}{2}) - (\alpha_1+\alpha_2)r$ and $B=1- \max(\alpha_2 r,\frac{3}{2})$.

If $\alpha_2 < \frac{3}{2r}$ we have $A=\frac{3}{2}-(\alpha_1+\alpha_2)r$ and $B= -\half$, so that
$$I^r  \lesssim |\xi|^{\alpha_0 r} ||\tau|-|\xi||^{\gamma r} |\xi|^{\frac{3}{2}-(\alpha_1+\alpha_2)r} ||\tau|-|\xi||^{-\half}  
 \lesssim |\xi|^{\alpha_0 r + \gamma r +1-(\alpha_1+\alpha_2)r} = 1 \, ,$$
where we used $|\tau| =||\eta|-|\xi-\eta|| \le |\xi|$ , $\gamma \ge \frac{1}{2r}$ and $\alpha_1+\alpha_2-\alpha_0 = \gamma + \frac{1}{r}$ .

If $\alpha_2 > \frac{3}{2r}$ we have $A=-\alpha_1 r$ and $B= 1-\alpha_2 r$. Thus
$$I^r  \lesssim |\xi|^{\alpha_0 r} ||\tau|-|\xi||^{\gamma r} |\xi|^{-\alpha_1 r} ||\tau|-|\xi||^{1-\alpha_2 r}  
 \lesssim |\xi|^{\alpha_0 r + \gamma r +1 -\alpha_2 r- \alpha_1 r} = 1 \, ,$$
using $|\tau| \le |\xi|$ , $\gamma \ge \alpha_2 - \frac{1}{r}$ and $\alpha_1+\alpha_2-\alpha_0 = \gamma + \frac{1}{r}$ . \\
{\bf b.} $|\eta|+|\xi-\eta| \ge 2|\xi|$ . By \cite{FK}, Lemma 4.4 we obtain : 
\begin{align*}
 &\int_{||\eta|+|\xi-\eta| \ge 2|\xi|} \delta(\tau-|\eta|+|\xi-\eta|) |\eta|^{-\alpha_1 r} |\xi - \eta|^{-\alpha_2 r} d\eta \\
& \sim (|\xi|^2-\tau^2)^{-\half} \int_2^{\infty} (|\xi|x+\tau)^{-\alpha_1 r} (|\xi|x-\tau)^{-\alpha_2 r} (|\xi|^2 x^2 - \tau^2)(x^2 -1)^{-\half} dx \\
& \sim (|\xi|^2-\tau^2)^{-\half} \int_2^{\infty} (x+\frac{\tau}{|\xi|})^{-\alpha_1 r+1} (x-\frac{\tau}{|\xi|})^{-\alpha_2 r+1} (x^2 -1)^{-\half} dx \,|\xi|^{-(\alpha_1+\alpha_2)r+2} \, .
\end{align*}
The lower limit of the integral can be chosen as $2$ by inspection of the proof of \cite{FK}, as already \cite{GN} remarked. If we now use our asumption $\alpha_1+\alpha_2 > \frac{2}{r}$ and using $|\tau| \le |\xi|$ the integral is obviously bounded. This implies
$$ I^r \lesssim |\xi|^{\alpha_0 r} ||\tau|-|\xi||^{\gamma r} \frac{|\xi|^{-(\alpha_1+\alpha_2)r+2}}{||\tau|-|\xi||^{\half} (|\tau|+|\xi|)^{\half}} \lesssim 1 
  $$
by our assumptions $\gamma \ge \frac{1}{2r}$ and $\alpha_1+\alpha_2-\alpha_0 = \gamma + \frac{1}{r}$ , and $|\tau| \le |\xi|$ .
\end{proof}

\begin{Cor}
\label{Cor.2}
Assume $1 < r \le 2$ , $\alpha_1+\alpha_2-\alpha_0 \ge \gamma + \frac{1}{r}$ , $\alpha_0 > \frac{1}{r}-\gamma$ , $0 \le \alpha_0 \le \alpha_1,\alpha_2$ , $\alpha_1 + \alpha_2 > \frac{2}{r}$ , $ \gamma \ge \max(\frac{1}{2r},\alpha_1-\frac{1}{r},\alpha_2-\frac{1}{r})$ and $\max(\alpha_1,\alpha_2) \neq \frac{3}{2r}$. 
Then the following estimate holds for $b > \frac{1}{r}$:
$$ \|D^{\alpha_0} B^{\gamma}_{\pm} (u,v)\|_{X^r_{0,0}} = \|{\mathcal F} (D^{\alpha_0} B^{\gamma}_{\pm} (u,v)\|_{L^{r'}_{\tau \xi}} \lesssim \|u\|_{X^r_{\alpha_1,b}} \|v\|_{X^r_{\alpha_2,b}} \, . $$
\end{Cor}
\begin{proof}
This is a consequence of the transfer principle Prop. \ref{Prop.0.1}.
\end{proof}

In the following we repeatedly use the following bilinear estimates in wave-Sobolev spaces $H^{s,b}$ , which were proven by d'Ancona, Foschi and Selberg in the two-dimensional case $n=2$ in \cite{AFS} in a more general form which include many limit cases which we do not need.
\begin{prop}
\label{AFS}
Let $n=2$. The estimate
$$\|uv\|_{H^{-s_0,-b_0}} \lesssim \|u\|_{H^{s_1,b_1}} \|v\|_{H^{s_2,b_2}} $$ 
holds, provided the following conditions hold:
\begin{align*}
\nonumber
& b_0 + b_1 + b_2 > \frac{1}{2} \\
\nonumber
& b_0 + b_1 \ge 0 \\
\nonumber
& b_0 + b_2 \ge 0 \\
\nonumber
& b_1 + b_2 \ge 0 \\
\nonumber
&s_0+s_1+s_2 > \frac{3}{2} -(b_0+b_1+b_2) \\
\nonumber
&s_0+s_1+s_2 > 1 -\min(b_0+b_1,b_0+b_2,b_1+b_2) \\
\nonumber
&s_0+s_1+s_2 > \frac{1}{2} - \min(b_0,b_1,b_2) \\
\nonumber
&s_0+s_1+s_2 > \frac{3}{4} \\
 &(s_0 + b_0) +2s_1 + 2s_2 > 1 \\
\nonumber
&2s_0+(s_1+b_1)+2s_2 > 1 \\
\nonumber
&2s_0+2s_1+(s_2+b_2) > 1 \\
\nonumber
&s_1 + s_2 \ge \max(0,-b_0) \\
\nonumber
&s_0 + s_2 \ge \max(0,-b_1) \\
\nonumber
&s_0 + s_1 \ge \max(0,-b_2)   \, .
\end{align*}
\end{prop}

\begin{lemma}
\label{Lemma1.3}
Assume $1\le r\le 2$ , $\alpha_1,\alpha_2 \ge \alpha_0 \ge 0$ ,  $b \ge \gamma$, and  $\alpha_1+\alpha_2-\alpha_0 \ge \gamma + \frac{1}{r}$, and $\gamma \ge \frac{1}{2r}$, where at least one of the last two inequalities is strict, and $b > \frac{1}{r}$ . Then the following estimate applies: 
$$ \|D^{\alpha_0} ((D^{\gamma}_- u) v))\|_{X^r_{0,0}} \lesssim \|u\|_{X^r_{\alpha_1,b}} \|v\|_{X^r_{\alpha_2,b}} \, , $$
where $D_-$ is the operator with Fourier symbol $\,||\tau|-|\xi||$.
\end{lemma}
\begin{proof}
By the fractional Leibniz rule and symmetry this reduces to 
$$\| (D^{\gamma}_- u) v\|_{X^r_{0,0}} \lesssim \|u\|_{X^r_{\alpha_1-\alpha_0,b}} \|v\|_{X^r_{\alpha_2,b}} \, .$$
This follows from 
$$\| u v\|_{X^r_{0,0}} \lesssim \|u\|_{X^r_{\alpha_1-\alpha_0,b-\gamma}} \|v\|_{X^r_{\alpha_2,b}} \, .$$

For $r=2$ we apply Prop. \ref{AFS} with parameters $s_0=0$ , $b_0=0$ , $s_1=\alpha_1-\alpha_0$, $b_1=b-\gamma$ , $s_2=\alpha_2$ , $b_2=b$ , so that $s_0+s_1+s_2 =\alpha_1+\alpha_2-\alpha_0 > \gamma + \half \ge \frac{3}{4} $ for $\gamma \ge \frac{1}{4}$ or $s_0+s_1+s_2 \ge \gamma + \half > \frac{3}{4}$ for $\gamma > \frac{1}{4}$ , and $s_0+s_1+s_2 > 1-b+\gamma$ for $\gamma \ge \frac{1}{4}$ and $b > \half$ . Moreover $s_0+s_1+s_2+s_0+s_2+b_1 \ge \gamma + \half + \alpha_2 + b-\gamma = \half + \alpha_2 + b > 1$ for $b>\half$ , $\alpha_2 \ge 0$. The remaining conditions which are needed for an application of \cite{AFS} are also easily checked.

In the case $\gamma = \frac{1}{2r}$ we obtain the claimed estimate by Prop. \ref{Prop.2.3} below with $b_1=b-\gamma > \frac{1}{2r}$ , $b_2=b > \frac{1}{r}$ , so that $b_1+b_2 > \frac{3}{2r}$. We also need the assumption $\alpha_1-\alpha_0+\alpha_2 > \gamma + \frac{1}{r}= \frac{3}{2r}$ .

In the case $\gamma = \frac{1}{2r}+$ an application of Prop. \ref{Prop.2.3} with $b_1 = b-\gamma > \frac{1}{2r}$ for $b=\frac{1}{r}++$ gives $b_1+b_2 > \frac{3}{2r}$  using the assumption  $\alpha_1-\alpha_0+\alpha_2 \ge \gamma + \frac{1}{r}> \frac{3}{2r}$ .

Next we consider the case $\gamma = b$ . We start with the subcase $r=1$ , so that $\alpha_1-\alpha_0+\alpha_2 \ge \gamma + \frac{1}{r} = b+\frac{1}{r} >\frac{2}{r}=2$ , because $b > \frac{1}{r}=1$. We obtain
\begin{align*}
 \|uv\|_{X^1_{0,0}} = &\,\|\tilde{uv}\|_{L^{\infty}_{\tau \xi}} \lesssim \|\tilde{u}\|_{L^{\infty}_\tau L^{q_1}_\xi} \|\tilde{v}\|_{L^{1}_\tau L^{q_2}_\xi} \\
\lesssim & \,\|\langle \xi \rangle^{-\alpha_1+\alpha_0} \|_{L^{\infty}_{\tau} L^{q_1}_{\xi}} \|\tilde{u}\|_{X^1_{\alpha_1-\alpha_0,0}} 
  \|\langle \xi \rangle^{-\alpha_2} \langle |\tau|-|\xi| \rangle^{-b} \|_{L^{1}_{\tau} L^{q_2}_{\xi}} \|\tilde{v}\|_{X^1_{\alpha_2,b}} \\
\lesssim &\, \|\tilde{u}\|_{X^1_{\alpha_1-\alpha_0,0}} \|\tilde{v}\|_{X^1_{\alpha_2,b}} \,.
\end{align*}
We choose  $\frac{1}{q_1}+\frac{1}{q_s} = 1 $ which fulfill the conditions $(\alpha_1-\alpha_0) q_1 > 2$ , $\alpha_2 q_2 > 2$ . This requires the condition  $\alpha_1+\alpha_2-\alpha_0 > 2$ , which is fulfilled.

Interpolation with the case $r=2$ , $\gamma = b$ gives the claimed result in the case
$1 < r \le 2$ , $\gamma =b$ ,  and another interpolation with the case $\gamma = \frac{1}{2r}+$ , which was treated  before, gives the final result in the case $1 < r \le 2$ , $b \ge \gamma > \frac{1}{2r}$ , if $\alpha_1+\alpha_2-\alpha_0 \ge \gamma + \frac{1}{r}$ .
\end{proof}

\begin{prop}
\label{Prop.1.4}
Assume $1 < r \le 2$ ,  $\alpha_0 > \frac{1}{r}-\gamma$ , $\alpha_1 + \alpha_2 > \frac{2}{r}$ , $0 \le \alpha_0 \le \alpha_1,\alpha_2$, $\max(\alpha_1,\alpha_2) \neq \frac{3}{2r}$ , $b \ge \gamma$ ,  and either $\alpha_1+\alpha_2-\alpha_0 > \gamma + \frac{1}{r}$ and $ \gamma \ge \frac{1}{2r}$ , or $\alpha_1+\alpha_2-\alpha_0 \ge \gamma + \frac{1}{r}$ and $ \gamma > \frac{1}{2r}$ . Moreover $ \gamma \ge \max(\alpha_1-\frac{1}{r},\alpha_2-\frac{1}{r})$ , $b > \frac{1}{r}$ .
Then the following estimate holds:
$$ \|D^{\alpha_0} D^{\gamma}_- (uv)\|_{X^r_{0,0}} \lesssim \|u\|_{X^r_{\alpha_1,b}} \|v\|_{X^r_{\alpha_2,b}} \, . $$
\end{prop}
\begin{proof}
We apply the "hyperbolic Leibniz rule" (cf. \cite{AFS}, p. 128):
$$ ||\tau|-|\xi|| \lesssim ||\rho|-|\eta|| + ||\tau - \rho|-|\xi - \eta|| + b_{\pm}(\xi,\eta) \, $$
where $b_{\pm}$ was defined  in Prop. \ref{Prop.1.1}. This implies
\begin{align*}
&\|D^{\alpha_0} D^{\gamma}_-(uv)\|_{X^r_{0,0}} \\
&\,\,\lesssim \|D^{\alpha_0} B^{\gamma}_{\pm}(u,v)\|_{X^r_{0,0}} + \|D^{\alpha_0} ((D^{\gamma}_- u)v)\|_{X^r_{0,0}} + \|D^{\alpha_0}(u D^{\gamma}_- v)\|_{X^r_{0,0}} \, . 
\end{align*}
Now we apply Cor. \ref{Cor.2} to the first term on the r.h.s. and Lemma \ref{Lemma1.3} to the second and third term, and the result follows.
\end{proof}

In the case $\gamma =0$ we need the following non-trivial result.
\begin{prop}
\label{Prop.2.3}
Let $ 1 \le r \le 2$ , $\alpha_1,\alpha_2 \ge 0$ , $\alpha_1 + \alpha_2 > \frac{3}{2r}$ , $b_1+b_2 > \frac{3}{2r}$ and $b_1,b_2 > \frac{1}{2r}$ . Then the following estimate holds
$$ \|uv\|_{X^r_{0,0}} \lesssim \|u\|_{X^r_{\alpha_1,b_1}} \|v\|_{X^r_{\alpha_2,b_2}} \, . $$
\end{prop}
\begin{proof}
Selberg \cite{S} proved this in the case $r=2$ . The general case $1 < r \le 2$ was given by Grigoryan-Tanguay \cite{GT}, Prop. 3.1, but in fact the case $r=1$ is also admissible. More precisely the result follows from \cite{GT} after summation over dyadic pieces in a standard way.
\end{proof}

As a consequence of these results we obtain
\begin{prop}
\label{Prop.1.5}
Under the assumptions of Prop. \ref{Prop.1.4}
the following estimate holds:
$$ \|uv\|_{X^r_{\alpha_0,\gamma}} \lesssim \|u\|_{X^r_{\alpha_1,b}} \|v\|_{X^r_{\alpha_2,b}} \, . $$
\end{prop}
\begin{proof}
Prop. \ref{Prop.1.4} reduces matters to the estimate
$$ \|uv\|_{X^r_{\alpha_0,0}} \lesssim \|u\|_{X^r_{\alpha_1,b}} \|v\|_{X^r_{\alpha_2,b}} \, , $$
which by the fractional Leibniz rule follows from
$$ \|uv\|_{X^r_{0,0}} \lesssim \|u\|_{X^r_{\alpha_1-\alpha_0,b}} \|v\|_{X^r_{\alpha_2,b}} \, . $$
We have by assumption either $\alpha_1 - \alpha_0 + \alpha_2 > \gamma + \frac{1}{r} \ge \frac{3}{2r}$ or   $\alpha_1 - \alpha_0 + \alpha_2 \ge \gamma + \frac{1}{r} > \frac{3}{2r}$,  $2b > \frac{2}{r} $ , $ b > \frac{1}{r}$, so that we may apply Prop. \ref{Prop.2.3} to complete the proof.
\end{proof}

We also need the following estimate for the angle $\angle(\xi_1,\xi_2)$ between two vectors $\xi_1$ and $\xi_2$ (cf. e.g. \cite{ST}, Lemma 4.3):
\begin{lemma}
\label{Lemma9.1}
Let $\tau_{\mu} \in \R$ , $\xi_{\mu} \in \R^2$ . If $\tau_0 =\tau_1 + \tau_2$ and $\xi_0 = \xi_1 + \xi_2$ the following estimate holds:
$$ \angle (\pm_1 \xi_1,\pm_2 \xi_2) \lesssim \left( \frac{\langle |\tau_0|-|\xi_0|\rangle + \langle |\tau_1|\pm_1|\xi_1|\rangle + \langle |\tau_2|\pm_2|\xi_2|\rangle}{\min(\langle \xi_1\rangle,\langle \xi_2 \rangle)} \right)^{\half} \, . $$
\end{lemma}

\section{The Chern-Simons-Higgs system (Proof of Theorem \ref{Thm.0.1})}
Conveniently we assume $\kappa=1$ .
In the Lorenz gauge $\partial^{\mu} A_{\mu} =0$ the Chern-Simons-Higgs system
 (\ref{0.1}),(\ref{0.2}) is equivalent to the following system
\begin{align}
\label{3.1'}
\Box \,  A_{\nu}  & = 2 \epsilon_{\mu\nu\lambda} \partial^{\mu} Im(\overline{\phi} D^{\lambda} \phi) \\
\label{3.2'}
(\Box +1)\phi  &= 2i A^{\mu} \partial_{\mu} \phi  + A^{\mu} A_{\mu} \phi - \phi V'(|\phi|^2) + \phi\, .
\end{align}
The initial data are given by
$$
A_{\nu}(0) = a_{\nu} \, , \, \partial_t A_0(0) = -\partial^l a_l \, , \, \partial_t A_j(0) = \partial_j a_0 + 2 \epsilon_{0jk} Im(\overline{f}(\partial^k-ia^k)f) \, ,$$
$$\phi(0) = f \, , \, \partial_t \phi(0) = g \, .  $$

Using \cite{HO} we may reformulate the CSH system as 
\begin{align*}
(-i\partial_t \pm D)A_{\nu \pm} & = - \epsilon_{\mu \nu \lambda}R^{\mu}_{\pm} Im(\overline{\phi} D^{\lambda} \phi) \\
(-i \partial_t \pm \langle D \rangle) \phi_{\pm} & = \half \langle D \rangle^{-1} (2i A_{\mu} \partial^{\mu} \phi + A_{\mu} A^{\mu} \phi + \phi) \, .
\end{align*} 
The homogeneous parts are given by
\begin{align*}
A^{hom}_{0 \pm}(t) & = \half e^{\pm(-i)tD} (a_0 \pm \frac{1}{i} D^{-1}\partial^l a_l) \\
A^{hom}_{j\pm}(t) & = \half e^{\pm(-i)tD} (a_j \mp \frac{1}{i} D^{-1}\partial^j a_j) \\
\phi^{hom}_{\pm}(t) & = \half e^{\pm(-i)t\langle D \rangle} (f \pm \frac{1}{i} D^{-1} g) \, ,
\end{align*}
so that 
\begin{align*}
\|A^{hom}_{0\pm}\|_{X^r_{s,b,\pm}[0,T]} &\lesssim \sum_{\mu} \|a_{\mu}\|_{\hat{H}^{s,r} }\\ 
\|\phi_{\pm}^{hom}\|_{X^r_{s+\half,b,\pm}[0,T]} & \lesssim \|f\|_{\hat{H}^{s+\half,r}} + \|g\|_{\hat{H}^{s-\half,r}}
\end{align*}
with implicit constants independent of $T$ , where we used the well-known estimates (cf. e.g. \cite{G}) $\|e^{\pm(-i)tD} \phi_0\|_{X^r_{s,b,\pm}[0,T]} \lesssim \|\phi_0\|_{\hat{H}^{s,r}}$ .  Here $\phi=\phi_+ + \phi_-$ , $A_{\nu} = A_{\nu +} + A_{\nu -}$ and $\phi_{\pm} = \half(\phi \pm i^{-1} \langle D  \rangle^{-1} \partial_t \phi)$ , $A_{\nu\pm} = \half(A_{\nu} \pm i^{-1} D^{-1} \partial_t A_{\nu})$ . The Riesz transform is given by $R_{\pm}^j := \mp i^{-1} D^{-1} \partial_j$ and $R^0_{\pm} := -1$ . 

Using $D^{\lambda} = \partial^{\lambda} -i A^{\lambda}$ and the multilinear character of the nonlinearities the estimates concerning (\ref{3.1'}) are by Theorem \ref{Theorem0.3} reduced to
\begin{equation}
\label{3.1}
 \|\epsilon_{\mu \nu \lambda} R^{\mu}_{\pm} (Im(\overline{\phi_1} \partial^{\lambda} \phi_2))\|_{X^r_{s,b-1+}} \lesssim \sum_{\pm_1,\pm_2} \|\phi_{1,\pm_1}\|_{X^r_{s+\half,b,\pm_1}}  \|\phi_{2,\pm_2}\|_{X^r_{s+\half,b,\pm_2}} 
\end{equation}
and
\begin{align}
\nonumber
&\|\epsilon_{\mu \nu \lambda} R^{\mu}_{\pm} (Im(\overline{\phi_1} A^{\lambda} \phi_2))\|_{X^r_{s,b-1+}} \\
\label{3.2}
&\quad \quad\lesssim \sum_{\pm_1,\pm_2,\pm_3} \|\phi_{1,\pm_1}\|_{X^r_{s+\half,b,\pm_1}}  \|\phi_{2,\pm_2}\|_{X^r_{s+\half,b,\pm_2}} \|A_{\pm_3}\|_{X^r_{s,b,\pm_3}} \, . 
\end{align}
 (\ref{3.1}) is equivalent to the estimate
\begin{align*}
& \left|\int \epsilon_{\mu \nu \lambda} Im(\overline{\phi_{1,\pm_1}} R^{\lambda}_{\pm_2} \phi_{2,\pm_2}) R^{\mu}_{\pm_0} \phi_{0,\pm_0} \, dt\,dx \right|\\
& \quad\quad\lesssim \|\phi_{0,\pm_0}\|_{X^{r'}_{-s,1-b+}}  \|\phi_{1,\pm_1}\|_{X^{r}_{s+\half,b}}   \|\phi_{2,\pm_2}\|_{X^{r}_{s-\half,b}} \, ,
\end{align*}                                       
which contains the null form $Q^{\lambda \mu}$ between $\phi_2$ and $\phi_0$ , so that by duality we have to prove
\begin{equation}
\label{3.2''}
 \|Q^{\lambda \mu}_{\pm_1,\pm_2}(u,v)\|_{X^{r'}_{-s-\half,-b}} \lesssim \|u\|_{X^{r'}_{-s,1-b+}} \|v\|_{X^r_{s-\half,b}} 
\end{equation}
for $ s > \frac{3}{2r} - \half$ and $b=\half+\frac{1}{2r}+$ ,
where
\begin{align*}
 Q^{\lambda \mu}_{\pm_1,\pm_2} (u,v) & := R^{\lambda}_{\pm_1} u R^{\mu}_{\pm_2} v - R^{\mu}_{\pm_1} u R^{\lambda}_{\pm_2} v \, .
\end{align*}

The symbol of $Q^{\lambda \mu}$ is estimated using \cite{FK}, Lemma 13.2.
\begin{equation}
\label{Q1}
|\eta \times (\xi - \eta)| \lesssim |\eta|^{\half} |\xi-\eta|^{\half} (|\eta|+|\xi - \eta|)^{\half} (|\eta|+|\xi-\eta|-|\xi|)^{\half} 
\end{equation}
in the case of equal signs, and
\begin{equation}
\label{Q2}
|\eta \times (\xi - \eta)| \lesssim |\eta|^{\half} |\xi-\eta|^{\half} (|\eta|+|\xi - \eta|)^{\half} (|\xi|-||\eta| - |\xi-\eta||)^{\half} 
\end{equation}
in the case of unequal signs.

Thus (\ref{3.2''}) is reduced to
$$ \|uv\|_{X^{r'}_{-s-\half,-b+\half}} \lesssim \|u\|_{X^{r'}_{-s+\half,1-b+}} \|v\|_{X^r_{s-\half,b}}  $$
and 
$$ \|uv\|_{X^{r'}_{-s-\half,-b+\half}} \lesssim \|u\|_{X^{r'}_{-s,1-b+}} \|v\|_{X^r_{s,b}} \, .  $$
This is by duality equivalent to
\begin{equation}
\label{3.3}
 \|vw\|_{X^{r}_{s-\half,b-1-}} \lesssim \|v\|_{X^{r}_{s-\half,b}} \|w\|_{X^r_{s+\half,b-\half}}  \end{equation}
and
\begin{equation}
\label{3.3'}
 \|vw\|_{X^{r}_{s,b-1-}} \lesssim \|v\|_{X^{r}_{s,b}} \|w\|_{X^r_{s+\half,b-\half}} \, . 
\end{equation}
Let $s=\frac{3}{2r}-\half+$ . For bigger $s$ the result then also holds. By the fractional Leibniz rule this reduces to the following estimates:
\begin{align*}
 \|vw\|_{X^{r}_{0,b-1-}}&\lesssim \|v\|_{X^{r}_{s-\half,b}} \|w\|_{X^r_{1,b-\half}} \, , \\
\|vw\|_{X^{r}_{0,b-1-}}& \lesssim \|v\|_{X^{r}_{s,b}} \|w\|_{X^r_{\half,b-\half}} \, .
\end{align*}
For $r=1+$ , $b = \half + \frac{1}{2r}+$ this holds by Prop. \ref{Prop.2.3}, because $s+\half > \frac{3}{2r}$ , $2b-\half > \half+\frac{1}{r} > \frac{3}{2r}$ 
and $b-\half > \frac{1}{2r}$ , and for $r=2$ , $b=\frac{3}{4}+$ and $s > \frac{1}{4}$ by Prop. \ref{AFS} , thus by interpolation for $1<r\le 2$ .

Next we consider the quadratic term in (\ref{3.2'}). We want to show
\begin{equation}
\label{4.1}
\|A^{\mu} \partial_{\mu} \phi\|_{X^r_{s-\half,b-1+,\pm}} \lesssim \sum_{\pm_1,\pm_2} \|A_{\pm_1}\|_{X^r_{s,b,\pm_1}} \|\phi_{\pm_2}\|_{X^r_{s+\half,b,\pm_2}} \, .
\end{equation}
We decompose
$ \partial_j  \phi = \sum_{\pm} (\pm i) R_{\pm}^j D \phi_{\pm}$ , $\partial_0 \phi = \sum_{\pm} (\pm i) R_{\pm}^0 D \phi_{\pm} - \sum_{\pm} (\pm i)(D-\langle D \rangle)\phi_{\pm}$ . The last term is easily treated, because $|\xi - \langle \xi \rangle| \le 1$ , so that 
in fact we only have to prove in this case
$$ \|uv\|_{X^r_{s-\half,b-1+}} \lesssim \|u\|_{X^r_{s,b}} \|v\|_{X^r_{s+\half,b}} \, ,$$ 
 which is obviously true, if (\ref{4.4}) below holds.

We now use the Hodge decomposition $A_i = A_i^{df} + A_i^{cf} $ , where the divergence-free part $A^{df}$ and the curl-free part $A^{cf}$ are given by
\begin{align*}
A^{df} & = (R_1 R_2 A_2-R_2 R_2 A_1, R_1R_2A_1-R_1R_1 A_2) \, , \\
A^{cf} & = (-R_1R_2 A_2-R_1R_1A_1,-R_1R_2A_1-R_2R_2A_2) \, ,
\end{align*}
where $R_j = i^{-1} \langle D \rangle^{-1} \partial_j$ . It is well-known (cf. \cite{ST} or \cite{HO}) that 
$$ \sum_{\pm_1} A^{df}_l R^l_{\pm_1} \psi_{\pm_1} = \sum_{\pm_1,\pm_2} Q^{12}_{\pm_1,\pm_2}(B_{\pm_2},\psi_{\pm_1}) \, , $$
where $B_{\pm}:= R_{\pm_1} A_{2 \pm} - R_{\pm_2} A_{1 \pm} $ , and
$$ \sum_{\pm_1} (A_0R^0_{\pm_1} \psi_{\pm_1} + A^{cf}_l R^l_{\pm_1} \psi_{\pm_1}) = \sum_{\pm_1,\pm_2} Q^0_{\pm_1,\pm_2}(\psi_{\pm_1},A_{0 \pm_2}) \, .$$ 
The null form $Q^{12}$ was defined before and $Q^0$ is defined by
\begin{align*}
 Q^0 _{\pm_1,\pm_2} (u,v) & = R_{\mu \pm_1} u R^{\mu}_{\pm_2} v \, .
\end{align*}
What remains to be proven are the following estimates:
\begin{equation}
\label{4.2}
\| Q^{12}_{\pm_1,\pm_2} (u,v)\|_{X^r_{s-\half,b-1+}} \lesssim \|u\|_{X^r_{s,b}} \|v\|_{X^r_{s-\half,b}} 
\end{equation}
and
\begin{equation}
\label{4.3}
\| Q^0_{\pm_1,\pm_2} (u,v)\|_{X^r_{s-\half,b-1+}} \lesssim \|u\||_{X^r_{s,b}} \|v\|_{X^r_{s-\half,b}} \, .
\end{equation}
As before we handle (\ref{4.2}) by \cite{FK}, Lemma 13.2 which reduces matters to
\begin{equation}
\label{4.4}
\|uv\|_{X^r_{s,b-\half+}} \lesssim \|u\|_{X^r_{s+\half,b}} \|v\|_{X^r_{s,b}} \, .
\end{equation}
We apply Prop. \ref{Prop.1.5} with parameters $\alpha_0 =s$ , $\alpha_1=s+\half$ , $\alpha_2 = s$ , $\gamma = s+\half-\frac{1}{r} $, thus $\alpha_1+\alpha_2-\alpha_0 = s+\half = \gamma + \frac{1}{r} $ , $\alpha_1+\alpha_2 = 2s+\half > \frac{3}{r} -\half \ge \frac{2}{r}$ , $\gamma > \frac{1}{2r}$ by our assumption $ s > \frac{3}{2r} -\half $ , and  $\gamma = \alpha_1-\frac{1}{r}$ . We also remark that $\gamma > b-\half+$ under our assumption $b= \half+\frac{1}{2r}+$ . This implies (\ref{4.4}).

In order to prove (\ref{4.3}) we use Lemma \ref{Lemma9.1} and the fact that the symbol of $Q^0_{\pm_1,\pm_2}$ is bounded by the angle $\angle(\pm_1 \xi_1,\pm_2 \xi_2)$ between the vectors $\pm_1\xi_1$ and $\pm_2\xi_2$ (even  $\angle(\pm_1 \xi_1,\pm_2 \xi_2)^2$) (cf. e.g. \cite{HO}, Lemma 2.6). Thus (\ref{4.3}) reduces to the following five estimates:
\begin{align}
\label{4.1'}
\|uv\|_{X^r_{s-\half,b-\half+}} & \lesssim \|u\|_{X^r_{s+\half,b}} \|v\|_{X^r_{s-\half,b}} \\
\label{4.2'}
\|uv\|_{X^r_{s-\half,b-\half+}} & \lesssim \|u\|_{X^r_{s,b}} \|v\|_{X^r_{s,b}} \\
\label{4.3'}
\|uv\|_{X^r_{s-\half,b-1+}} & \lesssim \|u\|_{X^r_{s+\half,b-\half}} \|v\|_{X^r_{s-\half,b}} \\
\label{4.4'}
\|uv\|_{X^r_{s-\half,b-1+}} & \lesssim \|u\|_{X^r_{s,b-\half}} \|v\|_{X^r_{s,b}} \\
\label{4.5'}
\|uv\|_{X^r_{s-\half,b-1+}} & \lesssim \|u\|_{X^r_{s+\half,b}} \|v\|_{X^r_{s-\half,b-\half}} \, .
\end{align}
We first consider (\ref{4.1'}) and (\ref{4.2'}) in the case $r=1+$. We apply Prop. \ref{Prop.1.5} with $\alpha_0=s-\half$ , $\alpha_1= s+\half$ , $\alpha_2= s-\half$ (or $\alpha_1= \alpha_2=s$) , $\gamma=s+\half-\frac{1}{r}$ assuming $s=\frac{3}{2r}-\half+\epsilon$ and $b=\half+\frac{1}{2r}+$ . We easily check the necessary conditions: one has $\alpha_1+\alpha_2-\alpha_0 =s+\half= \gamma + \frac{1}{r}$ , $\alpha_1+\alpha_2 = 2s = \frac{3}{r}-1+2\epsilon > \frac{2}{r}$ for $r$ close enough to $1$. Moreover $\alpha_0 = \frac{3}{2r}-1+\epsilon > \frac{1}{2r}-\epsilon = \frac{2}{r}-s-\half = \frac{1}{r}-\gamma$ and $\gamma = \frac{1}{2r}+\epsilon$ as well as $\gamma= \alpha_1-\frac{1}{r}$ . Because $\gamma > b-\half$ this implies (\ref{4.1'}) and (\ref{4.2'}) in the case $r=1+$. In the case $r=2$ , $ s=\frac{1}{4}+\epsilon$ , $b=\frac{3}{4}+ $ we apply Prop. \ref{AFS} with $s_0=\half -s= \frac{1}{4}-\epsilon$, $b_0=-\frac{1}{4}-$ , $s_1=s+\half$ , $s_2=s-\half$ , $b_1=b_2 = \frac{3}{4}+$ (or $s_1=s_2=s$) , so that $s_0+s_1+s_2 = \frac{3}{4}+\epsilon$ and $s_0+b_0+2s_1+2s_2= s+\half+2s-\frac{1}{4}- > 1$ .
By interpolation this implies (\ref{4.1'}) and (\ref{4.2'}) in the general case $1<r\le2$ .

Next we treat (\ref{4.3'}). In the case $r=1+$ we use the fractional Leibniz rule first and then Prop. \ref{Prop.2.3} with $\alpha_1=1$ , $\alpha_2 = s-\half$ (or $\alpha_1=s+\half$ , $\alpha_2=0$) , $b_1=b-\half$, $b_2=b$  , so that $\alpha_1+\alpha_2 =s+\half > \frac{3}{2} > \frac{3}{2r}$ , $b_1+b_2=2b-\half = \half + \frac{1}{r} + > \frac{1}{2r} $, $b_1 > \frac{1}{2r}$ . Similarly (\ref{4.4'}) and (\ref{4.5'}) can be treated. In the case $r=2$ , $s= \frac{1}{4}+\epsilon$ and $b=\frac{3}{4}+$ we apply Prop. \ref{AFS}. For (\ref{4.3'}) we have $s_0=\half-s= \frac{1}{4}-\epsilon$ , $b_0=\frac{1}{4}-$ , $s_1=s+\half=\frac{3}{4}+\epsilon$ , $b_1= b-\half= \frac{1}{4}+$ , $s_2= s-\half= -\frac{1}{4}+\epsilon$ , $b_2=b=\frac{3}{4}+$ , so that
$s_0+s_1+s_2= \frac{3}{4}+\epsilon$ , $s_0+s_1+s_2+s_0+s_2+b_1 = \frac{3}{4}+\epsilon+\frac{1}{4}+ > 1$ , thus (\ref{4.3'}) holds. Similarly (\ref{4.4'}) and (\ref{4.5'}) are proven for $r=2$ . By bilinear interpolation we obtain these estimates for $1<r \le 2$ .

Next we prove the estimates for the cubic terms starting with (\ref{3.2}). We obtain
$$
\|A\phi_1 \phi_2\|_{X^r_{0,0}}  \lesssim \|A\|_{X^r_{0,b}} \|\phi_1 \phi_2\|_{X^r_{s+\half,s+\half-\frac{1}{r}}} 
\lesssim \|A\|_{X^r_{0,b}} \|\phi_1\|_{X^r_{s+\half,b}} \|\phi_2\|_{X^r_{s+\half,b}} \, ,
$$
where we used $s+\half > \frac{3}{2r}$ and $b+s+\half-\frac{1}{r} > \half+\frac{1}{r} > \frac{3}{2r}$ , so that Prop. \ref{Prop.2.3} applies for the first estimate, and the fractional Leibniz rule combined with (\ref{4.1'}) for the second estimate. Moreover similarly we obtain
$$\|A\phi_1 \phi_2\|_{X^r_{0,0}}  \lesssim \|A\|_{X^r_{s,b}} \|\phi_1 \phi_2\|_{X^r_{\half,s+\half-\frac{1}{r}}} 
\lesssim \|A\|_{X^r_{s,b}} \|\phi_1\|_{X^r_{s+\half,b}} \|\phi_2\|_{X^r_{\half,b}} \, .
$$
The first estimate follows as above and the second estimate by Prop. \ref{Prop.1.5} with parameters $\alpha_0 = \half$ , $\alpha_1=s+\half$ , $\alpha_2 = \half$ , $\gamma = s+\half-\frac{1}{r}$ , so that $\alpha_1 + \alpha_2 = s+1 = \frac{3}{2r}+\half+ > \frac{2}{r}$ , $\alpha_1+\alpha_2-\alpha_0 = s+\half = \gamma+\frac{1}{r}$ , $\alpha_0=\half > \frac{1}{r}-\gamma = \frac{1}{2r}-$ and $\gamma > \frac{1}{2r}$ .
Consequently by the fractional Leibniz rule we obtain
$$ \|A\phi_1 \phi_2\|_{X^r_{s,0}} 
\lesssim \|A\|_{X^r_{s,b}} \|\phi_1\|_{X^r_{s+\half,b}} \|\phi_2\|_{X^r_{s+\half,b}} \, ,
$$
which implies (\ref{3.2}).

Next we have to consider the cubic term $A^{\mu} A_{\mu} \phi$ in equation (\ref{3.2'}), which requires the estimate
\begin{equation}
\label{6.1}
\|A_1 A_2 \phi\|_{X^r_{s-\half,b-1+}} \lesssim \|A_1\|_{X^r_{s,b}} \|A_2\|_{X^r_{s,b}} \|\phi\|_{X^r_{s+\half,b}} \, . 
\end{equation}
First we consider the case $r=1+$ , $b > \frac{1}{r}$ , $s > 1$ . By the fractional Leibniz rule we reduce to
\begin{equation}
\label{6.2}
\|A_1 A_2 \phi\|_{X^r_{0,0}} \lesssim \|A_1\|_{X^r_{\half,b}} \|A_2\|_{X^r_{s,b}} \|\phi\|_{X^r_{s+\half,b}} \, . 
\end{equation}
and
\begin{equation}
\label{6.3}
\|A_1 A_2 \phi\|_{X^r_{0,0}} \lesssim \|A_1\|_{X^r_{s,b}} \|A_2\|_{X^r_{s,b}} \|\phi\|_{X^r_{1,b}} \, . 
\end{equation}
We obtain
$$
\|A_1 A_2 \phi\|_{X^r_{0,0}} \lesssim \|A_1\|_{X^r_{\half,b}} \|A_2 \phi\|_{X^r_{1,s+\half-\frac{1}{r}}} 
 \lesssim \|A_1\|_{X^r_{\half,b}} \|A_2\|_{X^r_{1,b}} \|\phi\|_{X^r_{s+\half,b}} 
$$
by Prop. \ref{Prop.2.3} for the first step and Prop. \ref{Prop.1.5} for the second step with parameters $\alpha_0 = \alpha_1=1$ , $\alpha_2= s+\half$ , $\gamma=s+\half-\frac{1}{r}$ , so that $\alpha_1+\alpha_2-\alpha_0 = \gamma + \frac{1}{r}$ , $\alpha_1+\alpha_2 > \frac{2}{r}$ , $\gamma > \frac{1}{2r}$ , $\gamma = \alpha_2-\frac{1}{r}$ and also $\gamma = s+\half-\frac{1}{r} \ge 1-\frac{1}{r}= \alpha_1-\frac{1}{r}$ , because $s \ge \half$ for $r=1+$ . In a similar way we obtain by Prop. \ref{Prop.2.3} and Prop. \ref{Prop.1.5} with $\alpha_0=\half$ , $\alpha_1=s$ , $\alpha_2=1$ and $\gamma =s+\half-\frac{1}{r}$ :
$$
\|A_1 A_2 \phi\|_{X^r_{0,0}}  \lesssim \|A_1\|_{X^r_{s,b}} \|A_2 \phi\|_{X^r_{\half,s+\half-\frac{1}{r}}} 
 \lesssim \|A_1\|_{X^r_{s,b}} \|A_2\|_{X^r_{s,b}} \|\phi\|_{X^r_{1,b}} \, ,
$$
which implies (\ref{6.3}) and also (\ref{6.1}) in the case $r=1+$ .
In the case $r=2$ , $s=\frac{1}{4}+\epsilon$ , $b= \frac{3}{4}+$ we use Prop. \ref{AFS} and obtain
$$
\|A_1 A_2 \phi\|_{X^2_{s-\half,-\frac{1}{4}+}} \lesssim \|A_1\|_{X^2_{s,b}} \|A_2 \phi\|_{X^2_{s,0}} 
 \lesssim \|A_1\|_{X^r_{s,b}} \|A_2\|_{X^2_{s,b}} \|\phi\|_{X^2_{s+\half,b}} \, .
$$
By trilinear interpolation between the cases $r=1+$ and $r=2$ this implies (\ref{6.1}) for $1<r\le 2$ .

Finally we consider the term $\phi V'(|\phi|^2)$ , where $V(r) = r(1-r)^2$ or more generally a cubic polynomial. This means that $\phi V'(|\phi|^2)$ is a quintic  polynomial with most critical term $\phi |\phi|^4$ . We have to prove for $s=\frac{3}{2r} -\half+\epsilon$ , $b=\half+\frac{1}{2r}+$, $\epsilon > 0$ small:
$$ \| \prod_{i=1}^5 u_i\|_{X^r_{s-\half,b-1+}} \lesssim \prod_{i=1}^5 \|u_i\|_{X^r_{s+\half,b}} \, . $$
In the case $r=2$ we apply Prop. \ref{AFS} repeatedly:
\begin{align*}
\|\prod_{i=1}^5 u_i\|_{H^{-\frac{1}{4}+\epsilon,-\frac{1}{4}+}} & \lesssim \|u_1\|_{H^{\frac{3}{4}+\epsilon,\frac{3}{4}+}} \|\prod_{i=2}^5 u_i\|_{H^{-\frac{1}{8}+,-\frac{1}{8}+}} \, ,
\\
\|\prod_{i=2}^5 u_i\|_{H^{-\frac{1}{8}+,-\frac{1}{8}+}} & \lesssim \|u_2\|_{H^{\frac{3}{4}+\epsilon,\frac{3}{4}+}} \|\prod_{i=3}^5 u_i\|_{H^{0,0}} \, ,
\\
\|\prod_{i=3}^5 u_i\|_{H^{0,0}} & \lesssim \|u_3\|_{H^{\frac{3}{4}+\epsilon,\frac{3}{4}+}} \|\prod_{i=4}^5 u_i\|_{H^{\frac{1}{8},\frac{1}{8}}} \, ,
\\
\|\prod_{i=4}^5 u_i\|_{H^{\frac{1}{8},\frac{1}{8}}} & \lesssim \|u_4\|_{H^{\frac{3}{4}+\epsilon,\frac{3}{4}+}} \|u_5\|_{H^{\frac{1}{4},\frac{1}{4}}} \, ,
\end{align*}
which is more than enough. If $r=1+$ and $s=1+\epsilon$ we obtain by Prop. \ref{Prop.2.3} and the fractional Leibniz rule
$$\|\prod_{i=1}^5 u_i\|_{X^r_{s-\half,0}}  \lesssim \|u_1\|_{X^r_{s+\half,b}} \|\prod_{i=2}^5 \|u_i\|_{X^r_{s-\half,\frac{1}{2r}+\epsilon}} \, . $$
The estimate
$$\|\prod_{i=2}^5 u_i\|_{X^r_{s-\half,\frac{1}{2r}+\epsilon}}  \lesssim \|u_2\|_{X^r_{s+\half,b}} \|\prod_{i=3}^5 \|u_i\|_{X^r_{s-\half,b}}  $$
is obtained by Prop. \ref{Prop.1.5} with $\alpha_0=s-\half$ , $\alpha_1=s+\half$ , $\alpha_2=s-\half$ and $\gamma= \frac{1}{2r}+\epsilon $, so that $\alpha_1+\alpha_2-\alpha_0 = s+\half= \gamma + \frac{1}{r}$ , $\alpha_0=\frac{3}{2r}-1+\epsilon > \frac{1}{r}-\gamma =  \frac{1}{2r}-\epsilon$ for $r$ close enough to $1$. Similarly we obtain
$$\|\prod_{i=3}^5 u_i\|_{X^r_{s-\half,b}}  \lesssim \|u_3\|_{X^r_{s+\half,b}} \|\prod_{i=4}^5 \|u_i\|_{X^r_{s,b}}  $$
by Prop. \ref{Prop.1.5} with $\alpha_0 = s-\half$ , $\alpha_1=s+\half$ , $\alpha_2=s$ and $\gamma = b$ , so that $\alpha_1+\alpha_2-\alpha_0 = \frac{3}{2r}+\half+\epsilon > \gamma+\frac{1}{r}=\frac{3}{2r}+\half+$ and also
$$\|\prod_{i=4}^5 u_i\|_{X^r_{s,b}}  \lesssim \|u_4\|_{X^r_{s+\half,b}} \| \|u_5\|_{X^r_{s+\half,b}}  \, .$$
By multilinear interpolation we obtain the desired estimate for $1<r\le 2$ .

\section{The Chern-Simons-Dirac system (Proof of Theorem \ref{Thm.0.2})}
Following \cite{AFS1} and \cite{HO} the Dirac operator is introduced. Let $\eta_{\mu \nu}$ be the Minkowski metric on $\R^{1+2}$ with signature $(+,-,-)$. The genuine
matrices $\gamma^{\mu}$ are defined by 
$ \gamma^0 = \sigma^3 \quad \gamma^1=i \sigma^2 \quad \gamma^2 = -i \sigma^1 \, , $
where $\sigma^j$ ($j=1,2,3$) are the Pauli matrices
$$ \sigma^1 = \left( \begin{array}{cc}
0 & 1  \\ 
1 & 0  \end{array} \right)  \, 
 , \, \sigma^2 = \left( \begin{array}{cc}
0 & -i  \\
i & 0  \end{array} \right) \, ,  \, \sigma^3 = \left( \begin{array}{cc}
1 & 0  \\
0 & -1  \end{array} \right) \, .$$ 
Let $\psi$ be a 2-spinor field, i.e. a ${\mathcal C}^2$ - valued function. The Dirac operator of mass $m \ge 0$ is defined by
$$ \mathcal{D} \psi = (i\gamma^{\mu} \partial_{\mu} -m) \psi \, . $$
Defining $\beta = \gamma^0$ , $\alpha_1= \gamma^0 \gamma^1 = \sigma^1$ , $ \alpha_2= \gamma^0\gamma^2 = \sigma^2 $ we obtain
$$ \beta {\mathcal D} \psi = i \partial_t \psi + i \alpha^j \partial_j \psi - m \beta \psi \, . $$
Defining the (2x2)-matrices $\Pi = \Pi(\xi)$ by
$$ \Pi(\xi) = \half(I_{2x2} + \frac{\xi_j \alpha^j}{|\xi|}) \,\quad \, \Pi_{\pm}(\xi) = \Pi(\pm \xi) $$ we obtain
$$ \Pi_{\pm}(\xi)^2 = \Pi_{\pm}(\xi) \, , \, \Pi_+(\xi) \Pi_-(\xi) =0\, , \, I_{2x2} = \Pi_+(\xi) + \Pi_-(\xi) \, , $$
$$ \xi_j \alpha^j = |\xi| \Pi_+(\xi) - |\xi| \Pi_-(\xi)  $$
and also
\begin{align}
\label{7.5}
&\Pi_{\pm}(\xi) = \Pi_{\mp}(-\xi) \quad , \quad \beta \Pi(\xi) = \Pi(-\xi) \beta \, , \\
\label{7.6}
& \alpha^j \Pi(\xi) = \Pi(-\xi) \alpha^j + \frac{\xi_j}{|\xi|} I_{2x2} \, . 
\end{align}
Defining  the operator $ \Pi_{\pm} = \Pi_{\pm}(\frac{\nabla}{i})$ and $\psi_{\pm} = \Pi_{\pm} \psi$ we obtain
$$ \Pi_{\pm}(\beta {\mathcal D} \psi) = -i(i\partial_t \pm D) \psi_{\pm} - m \beta \psi_{\mp} \, . $$
Defining the modified Riesz transforms $R^{\mu}_{\pm}$ by $R^0_{\pm} = -1$ , $R^j_{\pm} = \mp((iD)^{-1}\partial_j)$  the equation (\ref{7.6}) my be rewritten as
\begin{equation}
\label{7.7}
\alpha^{\mu} \Pi_{\pm} = \Pi_{\mp} \alpha^{\mu} \Pi_{\pm} - R^{\mu}_{\pm} \Pi_{\mp} \, .
\end{equation}

Using the matrices $\alpha^{\lambda} $ we can reformulate the CSD system (\ref{0.4}),(\ref{0.5}) as follows, assuming for convenience $\kappa=1$ :
\begin{align}
\label{3.1a}
\partial_{\mu} A_{\nu} -\partial_{\nu} A_{\mu}& = -2 \epsilon_{\mu \nu \lambda} \psi ^{\dag} \alpha^{\lambda} \psi \\
\label{3.1b}
i \partial_t \psi + i \alpha^j \partial_j \psi & = m \beta\psi - \alpha^{\mu} A_{\mu} \psi \, .
\end{align}
Imposing the Lorenz gauge condition $\partial^{\mu} A_{\mu} =0$ and taking $\partial^{\mu}$ of (\ref{3.1a}) this can be formulated as
\begin{align}
\label{3.3a}
\Box \, A_{\mu} & = \partial^{\nu} N_{\mu \nu}(\psi,\psi) \\
\label{3.3b}
i \partial_t \psi + i \alpha^j \partial_j \psi & = m \beta \psi + M(\psi,A) \, ,
\end{align}
where
\begin{align*}
N_{\mu\nu} (\psi_1,\psi_2) & = -2 \epsilon_{\mu \nu \lambda} \psi_1 ^{\dag} \alpha^{\lambda} \psi_2 \\
M(\psi,A) & = -\alpha^{\mu} A_{\mu} \psi \, .
\end{align*}
The initial data are given by
$$ A_{\mu}(0) = a_{\mu} \, , \, \psi(0) = \psi_0 \, , \, \partial_t A_0(0) = -\partial^l a_l \, , \, \partial_t A_j(0) = \partial_j a_0 - 2 \epsilon_{0jk} \psi_0 ^{\dag} \alpha^k \psi_0 \, . $$

By the results of \cite{HO} we may further reformulate this similarly as in the CSH case as
\begin{align*}
(-i \partial_t \pm D)  A_{\nu \pm} & = (2i)^{-1} R^{\mu}_{\pm} N_{\mu \nu} (\psi,\psi) \\
-(-i\partial_t \pm D) \psi_{\pm} & = m \beta \psi_{\mp} + \Pi_{\pm} M(\psi,A) \, .
\end{align*}
We recall that $\psi_{\pm} = \Pi_{\pm} \psi$ , and  $A_{\nu} = A_{\nu +} + A_{\nu -}$ , $A_{\nu \pm} = \half(A_{\nu} \pm i^{-1} D^{-1}\partial_t A_{\nu})$ . Their homogeneous parts are given by (cf. \cite{HO}) :
\begin{align*}
\psi_{\pm}^{hom}(t) & = e^{\pm(-i)tD} \psi_0 \\
A_{0 \pm}^{hom}(t) & = \half e^{\pm(-i)tD} (a_0 \pm (iD)^{-1} \partial^l a_l) \\
A_{j \pm}^{hom}(t) & = \half e^{\pm(-i)tD} (a_j \pm (iD)^{-1} \partial_j a_0)  \, ,
\end{align*}
so that
\begin{align*}
\|A^{hom}_{\nu \pm}\|_{X^r_{s,b,\pm}[0,T]} &\lesssim \sum_{\mu} \|a_{\mu}\|_{\hat{H}^{s,r} }\\ 
\|\phi_{\pm}^{hom}\|_{X^r_{s,b,\pm}[0,T]} & \lesssim \|\psi_0\|_{\hat{H}^{s+\half,r}}
\end{align*}
with implicit constants independent of $T$, where we used the well-known estimates (cf. e.g. \cite{G})
$$ \|e^{\pm(-i)tD} \phi_0\|_{X^r_{s,b,\pm}[0,T]} \lesssim \|\phi_0\|_{\hat{H}^{s,r}} $$
for $b >\frac{1}{r}$.

The following estimates are sufficient for our local well-posedness Theorem \ref{Thm.0.2} by Theorem \ref{Theorem0.3} taking into account the multilinear character of the nonlinearities:
\begin{align}
\label{7.1}
\| R^{\nu}_{\pm 0} N_{\mu \nu}(\psi_1,\psi_2)\|_{X^r_{s,b-1+,\pm 0}} & \lesssim \sum_{\pm 1, \pm 2} \|\psi_{1, \pm 1} \|_{X^r_{s,b,\pm 1}} \|\psi_{2, \pm 2} \|_{X^r_{s,b,\pm 2}} \\
\label{7.2}
\| \Pi_{\pm 0} M(\psi,A_{\nu})\|_{X^r_{s,b-1+,\pm 0}} & \lesssim \sum_{\pm 1, \pm 2} \|\psi_{1, \pm 1} \|_{X^r_{s,b,\pm 1}} \|\psi_{2, \pm 2} \|_{X^r_{s,b,\pm 2}} \, .
\end{align}
where $ \|A_{\pm}\|_{X^r_{s,b,\pm}} := \sum_{\nu} \|A_{\nu\pm}\|_{X^r_{s,b,\pm}}$ .

Applying the commutator relation (\ref{7.7}) we obtain
\begin{align*}
N_{\mu \nu}(\psi_1,\psi_2) & = -2 \sum_{\pm 1,\pm 2} \epsilon_{\mu \nu \lambda} (\psi_{1,\pm 1} ^{\dag} \Pi_{\mp}( \alpha^{\lambda} \psi_{2,\pm 2})) + 2 \sum_{\pm 1,\pm 2} \epsilon_{\mu \nu \lambda} (\psi_{1,\pm 1} ^{\dag} R^{\lambda}_{\mp 2} \psi_{2,\pm 2}) \\
& =: N_{\mu \nu,1} + N_{\mu \nu,2} \, .
\end{align*} 
The estimate for $N_{\mu \nu,2}$ in (\ref{7.1}) follows from
\begin{align*}
&\left| \int \epsilon_{\mu \nu \lambda}(\psi_{1,\pm 1} ^{\dag} R^{\lambda}_{\pm 2} \psi_{,\pm 2}) \overline{R^{\nu}_{\pm 0} \phi_{\pm 0}}\, dt\, dx \right| \\
&\quad \quad\lesssim \|\phi_{0,\pm 0}\|_{X^{r'}_{-s,1-b+,\pm 0}} \|\psi_{1,\pm 1}\|_{X^{r}_{s,b,\pm 1}} \|\psi_{2,\pm 2}\|_{X^{r}_{s,b,\pm 2}} \, .
\end{align*}
The left hand side contains the null form $Q^{\lambda \nu}$ between $\psi_{2,\pm 2}$ and $\phi_{0,\pm 0}$ . Thus by duality we reduce to
$$ \| Q^{\lambda \nu}_{\pm 1,\pm 2} (u,v)\|_{X^{r'}_{-s,-b}} \lesssim \|u\|_{X^{r'}_{-s,1-b+}} \|v\|_{X^r_{s,b}} \, .
$$
By (\ref{Q1}) and (\ref{Q2}) this is a consequence of the estimates
\begin{equation}
\label{8.1}
\|uv\|_{X^{r'}_{-s,\half-b}} \lesssim \|u\|_{X^{r'}_{-s+\half,1-b+}} \|v\|_{X^{r}_{s,b}} 
\end{equation}
and
\begin{equation}
\label{8.1'}
\|uv\|_{X^{r'}_{-s,\half-b}} \lesssim \|u\|_{X^{r'}_{-s,1-b+}} \|v\|_{X^{r}_{s+\half,b}} \, ,
\end{equation}
which by duality is equivalent to
$$ \|vw\|_{X^{r}_{s-\half,b-1-}} \lesssim \|v\|_{X^{r}_{s,b}} \|w\|_{X^{r}_{s,b-\half}}  $$
and
$$ \|vw\|_{X^{r}_{s,b-1-}} \lesssim \|v\|_{X^{r}_{s+\half,b}} \|w\|_{X^{r}_{s,b-\half}} \, . $$
In the case $r=1+$ , $1 \ge b > \frac{1}{r}$ and $s>1$ we use the fractional Leibniz rule and reduce to the following estimates
\begin{align*}
\|vw\|_{X^{r}_{0,b-1-}} &\lesssim \|v\|_{X^{r}_{\half,b}} \|w\|_{X^{r}_{s,b-\half}} \, , \\
 \|vw\|_{X^{r}_{0,b-1-}}&\lesssim \|v\|_{X^{r}_{s,b}} \|w\|_{X^{r}_{\half,b-\half}} \, , \\
 \|vw\|_{X^{r}_{0,b-1-}}&\lesssim \|v\|_{X^{r}_{s+\half,b}} \|w\|_{X^{r}_{0,b-\half}} \, ,
\end{align*}
which is implied by Prop. \ref{Prop.2.3} , thus (\ref{8.1}) and (\ref{8.1'}) are proven in this case. In the case $r=2$, $s= \frac{1}{4}+\epsilon$ and $b= \frac{3}{4}+$ we directly prove (\ref{8.1}) by Prop. \ref{AFS} with parameters $s_0=s=\frac{1}{4}+\epsilon$ , $b_0=b-\half=\frac{1}{4}+$ , $s_1=\half-s= \frac{1}{4}-\epsilon$ , $b_1= 1-b+ = \frac{1}{4}-$, $s_2=s=\frac{1}{4}+\epsilon$ , $b_2=b=\frac{3}{4}+$ , so that $s_0+s_1+s_2 = \frac{3}{4}+\epsilon$ . Similarly (\ref{8.1'}) can be proven with parameters as before but $s_1=-s=-\frac{1}{4}-\epsilon$ , $s_2 = s+\half= \frac{3}{4}+\epsilon$ . By bilinear interpolation we obtain (\ref{8.1}) in the range $1 < r \le 2$ .

The estimate for $N_{\mu \nu, 1}$ in (\ref{7.1}) is a consequence of \cite{AFS1}, Lemma 2, namely the bound
$$ |\Pi(\xi_1) \Pi(-\xi_2) z| \lesssim |z| \angle (\xi_1,\xi_2) \, , $$
which implies
\begin{align*}
& \|\epsilon_{\mu \nu \lambda} (\Pi_{\pm 1} \psi_{1,\pm 1} ^{\dag} \Pi_{\mp 2}( \alpha^{\lambda} \Pi_{\pm 2} \psi_{2,\pm 2}))\|_{X^r_{s,b-1+,\pm 0}} \\
& \quad \quad \quad \quad \quad \quad \lesssim \| \angle(\pm_1 \xi_1,\pm_2 \xi_2) \psi_{1, \pm 1} \psi_{2,\pm 2}\|_{X^r_{s,b-1+,\pm 0}} \, . 
\end{align*}
We apply Lemma \ref{Lemma9.1}. Thus we have to prove the following estimates
\begin{align}
\label{9.1}
\|uv\|_{X^r_{s,b-\half+}} & \lesssim \|u\|_{X^r_{s+\half,b}} \|v\|_{X^r_{s,b}} \\
\label{9.2}
\|uv\|_{X^r_{s,b-1+}} & \lesssim \|u\|_{X^r_{s+\half,b-\half}} \|v\|_{X^r_{s,b}} \\
\label{9.3}
\|uv\|_{X^r_{s,b-1+}} & \lesssim \|u\|_{X^r_{s,b-\half}} \|v\|_{X^r_{s+\half,b}} \, .
\end{align}
For $r=1+$ , $s>1$ , $b= \frac{1}{2r}+\half+$  we conclude as follows. (\ref{9.1}) is by the fractional Leibniz rule reduced to (\ref{4.1'}) or (\ref{4.2'}). (\ref{9.2}) and (\ref{9.3})  are implied by the fractional Leibniz rule combined with Prop. \ref{Prop.2.3}, where $b_1=b-\half$ , $b_2=b$, so that $b_1+b_2 > \frac{1}{r}+\half > \frac{3}{2r}$ . In the case $r=2$ , $s> \frac{1}{4}$ and $b= \frac{3}{4}+$  all estimates follow from Prop. \ref{AFS}. Interpolation completes the proof of (\ref{9.1}) - (\ref{9.3}) in the range $1<r \le 2$ and also (\ref{7.1}).

Now we come to the proof of (\ref{7.2}). We use the commutator formula  (\ref{7.7}) and obtain 
$$ - \Pi_{\pm}(A_{\mu} \alpha^{\mu} \psi) = M_{\pm,1}(\psi,A) +M_{\pm,2}(\psi,A) \, , $$
where
$$ M_{\pm 0,1}(\psi,A) = - \sum_{\pm} \Pi_{\pm 0}(A_{\mu} \Pi_{\mp} \alpha^{\mu} \psi_{\pm}) \; , \; M_{\pm 0,2}(\psi,A) = - \sum_{\pm} \Pi_{\pm 0}(A_{\mu} R^{\mu}_{\pm} \psi_{\pm}) \, . $$
The estimate for $M_{\pm,1}$ in (\ref{7.2}) follows from
\begin{align*}
&\left| \int A_{\mu,\pm 2} (\psi_{0,\pm 0} ^{\dag} \Pi_{\mp} \alpha^{\mu} \psi_{_1,\pm 1}) dt dx \right| \\
& \quad \lesssim \|\psi_{0,\pm 0}\|_{X^{r'}_{-s,1-b-,\pm 0}} \|\psi_{1,\pm 1}\|_{X^{r}_{s,b,\pm 1}} \|\psi_{2,\pm 2}\|_{X^{r}_{s,b,\pm 2}} \, .
\end{align*}
The left hand side contains a null form between $\psi_{0,\pm 0}$ and $\psi_{1,\pm 1}$ as for $N_{\nu \mu,1}$ . Thus it suffices to prove
$$ \| \angle(\pm_0 \xi_0,\pm_1 \xi_1) \psi_{0,\pm 0} \psi_{1,\pm 1}\|_{X^{r'}_{-s,-b,\pm 2}} \lesssim  \|\psi_{0,\pm 0}\|_{X^{r'}_{-s,1-b-,\pm 0}} \|\psi_{1,\pm 1}\|_{X^{r}_{s,b,\pm 1}} \, . $$
By Lemma \ref{Lemma9.1} we have to prove the following six estimates:
\begin{align}
\label{9.4}
\|uv\|_{X^{r'}_{-s,-b+\half}} & \lesssim \|u\|_{X^{r'}_{-s+\half,1-b-}} \|v\|_{X^r_{s,b}} \\
\label{9.5}
\|uv\|_{X^{r'}_{-s,-b+\half}} & \lesssim \|u\|_{X^{r'}_{-s,1-b-}} \|v\|_{X^r_{s+\half,b}} \\
\label{9.6}
\|uv\|_{X^{r'}_{-s,-b}} & \lesssim \|u\|_{X^{r'}_{-s+\half,\half-b-}} \|v\|_{X^r_{s,b}} \\
\label{9.7}
\|uv\|_{X^{r'}_{-s,-b}} & \lesssim \|u\|_{X^{r'}_{-s,\half-b-}} \|v\|_{X^r_{s+\half,b}} \\
\label{9.8}
\|uv\|_{X^{r'}_{-s,-b}} & \lesssim \|u\|_{X^{r'}_{-s+\half,1-b-}} \|v\|_{X^r_{s,b-\half}} \\
\label{9.9}
\|uv\|_{X^{r'}_{-s,-b}} & \lesssim \|u\|_{X^{r'}_{-s,1-b-}} \|v\|_{X^r_{s+\half,b-\half}} \, .
\end{align}
We first consider the case $r=1+$ , $s > 1$ , $b = \frac{1}{2r} + \half+$ . (\ref{9.5}) is by duality equivalent to
$$ \|uv\|_{X^{r}_{s,-1+b+}}  \lesssim \|u\|_{X^{r}_{s,b-\half}} \|v\|_{X^r_{s+\half,b}} \, . $$
Using $-1+b+ \le 0$ and the fractional Leibniz rule this follows from Prop. \ref{Prop.2.3}, because $s+\half > \frac{3}{2} > \frac{3}{2r}$ and $2b-\half > \frac{1}{r}+\half > \frac{3}{2r}$ . Similarly (\ref{9.4}) is equivalent to
$$ \|uv\|_{X^{r}_{s-\half,-1+b+}}  \lesssim \|u\|_{X^{r}_{s,b-\half}} \|v\|_{X^r_{s,b}} \, , $$
which also holds by Prop. \ref{Prop.2.3}. (\ref{9.6}) - (\ref{9.9}) are equivalent to
\begin{align}
\label{9.6'}
 \|uv\|_{X^{r}_{s-\half,b-\half+}} & \lesssim \|u\|_{X^{r}_{s,b}} \|v\|_{X^r_{s,b}} \\
\label{9.7'}
 \|uv\|_{X^{r}_{s,b-\half+}} & \lesssim \|u\|_{X^{r}_{s,b}} \|v\|_{X^r_{s+\half,b}} \\
\label{9.8'}
 \|uv\|_{X^{r}_{s-\half,b-1+}} & \lesssim \|u\|_{X^{r}_{s,b}} \|v\|_{X^r_{s,b-\half}} \\
\label{9.9'}
 \|uv\|_{X^{r}_{s,b-1+}} & \lesssim \|u\|_{X^{r}_{s,b}} \|v\|_{X^r_{s+\half,b-\half}} \, .
\end{align}
(\ref{9.6'}) is the same as (\ref{4.1'}), (\ref{9.7'}) is by the fractional Leibniz rule implied by (\ref{4.1'}) and (\ref{4.2'}), (\ref{9.8'}) and (\ref{9.9'}) follows from Prop. \ref{Prop.2.3}.

The case $r=2$ is handled by Prop. \ref{AFS}, then interpolation implies the whole range $1<r \le 2$ . 

Next we consider the term $M_{\pm 0,2}(\psi,A)$. It has the same structure as $A_{\mu} \partial^{\mu} \phi$ for the CSH system with $\partial^{\mu} \phi$ replaced by $R^{\mu}_{\pm} \psi_{\pm}$. Then using the Hodge decomposition $A_i = A_i^{df} + A_i^{cf}$ we reduce similarly as in the CSH - case to the following estimates.
\begin{align}
\label{11.1}
\|Q^{12}_{\pm 1,\pm 2} (u,v)\|_{X^r_{s,b-1+}} & \lesssim \|u\|_{X^r_{s,b}} \|v\|_{X^r_{s,b}} \, ,\\
\label{11.2}
\|Q^0_{\pm 1,\pm 2} (u,v)\|_{X^r_{s,b-1+}} & \lesssim \|u\|_{X^r_{s,b}} \|v\|_{X^r_{s,b}} \, .
\end{align}
By use of \cite{FK}, Lemma 13.2 the estimate (\ref{11.1}) is reduced to
$$ \|uv\|_{X^r_{s+\half,b-\half+}} \lesssim \|u\|_{X^r_{s+\half,b}} \|v\|_{X^r_{s+\half,b}} \, , $$
which by the fractional Leibniz rule is implied by (\ref{4.4}). The estimate (\ref{11.2}) follows from (\ref{4.3}) and the fractional Leibniz rule.


\begin{thebibliography}{999}
\bibitem[1]{AFS} P. d'Ancona, D. Foschi, and S. Selberg: {\sl Product estimates for wave-Sobolev spaces in 2+1 and 1+1 dimensions}. Contemp. Math. (2010) 526, 125-150
\bibitem[2]{AFS1}
P. d'Ancona, D. Foschi, and S. Selberg, {\sl Null structure and almost optimal local regularity for the Dirac-Klein-Gordon system}, J. Eur. Math. Soc.  (2007), no. 4, 877-899 Contemporary Math. 526 (2010), 125-150
\bibitem[3]{BCM} N. Bourneveas, T. Candy, and S. Machihara: {\sl A note on the Chern-Simons-Dirac equations in the Coulomb gauge}. Discr. Cont. Dyn. Syst. 34(7) (2014), 2693-2701
\bibitem[4]{CKP} Y.M. Cho, J.W. Kim and D.H. Park: {\sl Fermionic vortex solutions in Chern-Simons electrodynamics}. Phys. Rev. D 45 (1992), 3802-3806
\bibitem[5]{FK} D. Foschi and S. Klainerman: {\sl Bilinear space-time estimates for homogeneous wave equations.} Ann. Sc. ENS. 4. serie, 33 (2000), 211-274
\bibitem[6]{GN} V. Grigoryan and A. Nahmod: {\sl Almost critical wee-posedmess for nonlinear wave equation with $Q_{\mu \nu}$ null forms in 2D.} Math. Res. Letters 21 (2014), 313-332
\bibitem[7]{GT} V. Grigoryan and A. Tanguay: {sl Improved well-poseness for the quadratic derivative nonlinear wave equation in 2D.} Preprint arXiv:1308.1719
\bibitem[8]{G} A. Gr\"unrock: {\sl An improved local well-posedness result for the modified KdV equation.} Int. Math. Res. Not. (2004), no.61, 3287-3308
\bibitem[9]{G1} A. Gr\"unrock: {\sl On the wave equation with quadratic nonlinearities in three space dimensions.} Hyperbolic Diff. Equ. 8 (2011), 1-8
\bibitem[10]{GV} A. Gr\"unrock and L. Vega: {\sl Local well-posedness for the modified KdV equation in almost critical $\hat{H}^r_s$ -spaces.} Trans. Amer. Mat. Soc. 361 (2009), 5681-5694
\bibitem[11]{HKP} J. Hong, Y. Kim and P.Y. Pac: {\sl Multivortex solutions of the abelian Chern-Simons-Higgs theory}. Phys. Rev. Letters 64 (1990), 2230-2233
\bibitem[12]{H} H. Huh: {\sl Cauchy problem for the Fermion field equation coupled with the Chern-Simons gauge}. Lett. Math. Phys. 79 (2007), 75-94
\bibitem[13]{HO} H. Huh and S.-J. Oh: {\sl Low regularity solutions to the Chern-Simons-Dirac and the Chern-Simons-Higgs equations in the Lorenz gauge}. Comm. PDE 41(3) (2016), 375-397
\bibitem[14]{JW} R. Jackiw and E.J. Weinberg: {\sl Self-dual Chern-Simons vortices}. Phys. Rev. Letters 64 (1990), 2234-2237
\bibitem[15]{LB} S. Li and R.K. Bhaduri: {\sl Planar solitons of the gauged Dirac equation}. Phys. Rev. D 43 (1991), 3573-3574  
\bibitem[16]{P} H. Pecher: {\sl Low regularity solutions for Chern-Simons-Dirac systems in the temporal and  Coulomb gauge}. Electron. J. Differential Equations 2016(2016), No. 174, 1-16
\bibitem[17]{P1} H. Pecher: {\sl Global well-posedness in energy space for the Chern-Simons-Higgs system in temporal gauge.}  J. Hyperbolic Diff. Equ. 13(2016), 331-351
\bibitem[18]{S} S. Selberg: {\sl Bilinear Fourier restriction estimates related to the 2D wave equation.} Adv. Diff. Equ. 16 (2011), 667-690
\bibitem[19]{ST} S.Selberg and A. Tesfahun: {\sl Global well-posedness of the Chern-Simons-Higgs equations with finite energy}. Discrete Cont. Dyn. Syst. 33 (2013), 2531-2546
\bibitem[20]{VV} A. Vargas and L. Vega: {\sl Global wellposedness for 1D non-linear Schr\"odinger equation for data with an infinite $L^2$-norm.} J. Math. Pures Appl. 80 (2001), 1029-1044

\end{thebibliography}
\end{document}